\documentclass[a4paper,10pt]{article}
\usepackage{amsmath,amssymb}
\usepackage[colorlinks,citecolor=blue]{hyperref}
\usepackage[utf8]{inputenc}
\usepackage[all]{xy}
\usepackage{graphicx}
\usepackage{euler}
\usepackage{color}
\usepackage{skak}

\newcommand{\ZZ}{\mathbb{Z}}
\newcommand{\QQ}{\mathbb{Q}}

\newcommand{\bop}{\textsc{B}_{+}}
\newcommand{\ehr}{\operatorname{Ehr}}
\newcommand{\sv}{\mathrm{v}}
\renewcommand{\int}{\operatorname{Int}}

\newcommand{\pyr}{\operatorname{Pyr}}

\newcommand{\qbinom}[2]{\genfrac{[}{]}{0pt}{}{#1}{#2}_q}

\newtheorem{theorem}{Theorem}[section]
\newtheorem{proposition}[theorem]{Proposition}
\newtheorem{conjecture}[theorem]{Conjecture}

\newtheorem{lemma}[theorem]{Lemma}

\newtheorem{definition}{Definition}

\newtheorem{remark}[theorem]{Remark}
\newtheorem{example}[theorem]{Example}

\newenvironment{proof}{\begin{trivlist}\item{\bf{Proof.}}}
  {\hfill\rule{2mm}{2mm}\end{trivlist}}

\title{$q$-analogues of Ehrhart polynomials}
\author{F. Chapoton\footnote{This work has been supported by the
    ANR program CARMA (ANR-12-BS01-0017-02).}}
\date{\today}

\begin{document}

\maketitle

\begin{abstract}
  One considers weighted sums over points of lattice polytopes, where
  the weight of a point $v$ is the monomial $q^{\lambda(v)}$ for some
  linear form $\lambda$. One proposes a $q$-analogue of the classical
  theory of Ehrhart series and Ehrhart polynomials, including Ehrhart
  reciprocity and involving evaluation at the $q$-integers.
\end{abstract}

\section*{Introduction}

The theory of Ehrhart polynomials, which was introduced by Eug{\`e}ne
Ehrhart in the 1960s \cite{ehrhart}, has now become a classical
subject. Let us recall it very briefly. If $Q$ is a lattice polytope,
meaning a polytope with vertices in a lattice, one can count the
number of lattice points inside $Q$. It turns out that the number of
lattice points in the dilated lattice polytope $n Q$ for some integer
$n$ is a polynomial function of $n$. This is called the Ehrhart
polynomial of the lattice polytope $Q$. Moreover, the value of the
Ehrhart polynomial at a negative integer $-n$ is (up to sign) the
number of interior lattice points in $n Q$. This phenomenon is called
Ehrhart reciprocity. This classical theory is for example detailed in
the book \cite{beck}.

In this article, one introduces a $q$-analogue of this theory, in
which the number of lattice points is replaced by a weighted sum,
which is a polynomial in the indeterminate $q$. One proves that these
weighted sums for dilated polytopes are values at $q$-integers of a
polynomial in $x$ with coefficients in $\QQ(q)$. Let us present this
in more detail.

Let $Q$ be a lattice polytope and let $\lambda$ be a linear form on
the ambient lattice of $Q$, assumed to take positive values on
$Q$. One considers the weighted sum
\begin{equation}
  W_{\lambda}(Q,q) = \sum_{x \in Q} q^{\lambda(x)},
\end{equation}
running over lattice points in $Q$, where $q$ is an
indeterminate. This is a polynomial $q$-analogue of the number of
lattice points in $Q$, which is the value at $q=1$.

Under two hypotheses of positivity and genericity on the pair
$(Q,\lambda)$, one proves that the polynomials $W_{\lambda}(n Q,q)$
for integers $n\geq 0$ are the values at $q$-integers $[n]_q$ of a
polynomial in $x$ with coefficients in $\QQ(q)$, which is called the
$q$-Ehrhart polynomial.

One also obtains a reciprocity theorem, which relates the value of the
$q$-Ehrhart polynomial at the negative $q$-integer $[-n]_q$ to the
weighted sum over interior points in $n Q$.

In the special case where the lattice polytope is the order polytope
of a partially ordered set $P$, the theory presented here is closely
related to the well-known theory of $P$-partitions, introduced by
Richard P. Stanley in \cite{stanley_ordered}. One can find there
$q$-analogues of Ehrhart series, which coincide with the one used
here. It seems though that the existence of $q$-Ehrhart polynomials is
new even in this setting.

The $q$-Ehrhart polynomials seem to have interesting properties in the
special case of empty polytopes. In particular, they vanish at
$x=-1/q$ by reciprocity, and the derivative at this point may have a
geometric meaning. One also presents an umbral property, which
involves some $q$-analogues of Bernoulli numbers that were introduced
by Carlitz in \cite{carlitz}.

\medskip

Our original motivation for this theory came from the study of some
tree-indexed series, involving order polytopes of rooted trees, in the
article \cite{serieqx}. In this study appeared some polynomials in $x$
with coefficients in $\QQ(q)$, who become Ehrhart polynomials when
$q=1$. Understanding this has led us to the results presented here.

\medskip

Some other generalisations of the classical Ehrhart theory have been
considered in \cite{stap1,stap2}, but they do not seem to involve evaluation
at $q$-integers.

\medskip

The article is organised as follows. In section \ref{qseries}, one
introduces the general setting and hypotheses and then studies the
$q$-Ehrhart series. In section \ref{qpoly}, one proves the existence
of the $q$-Ehrhart polynomial and obtains the reciprocity theorem. In
section \ref{otherprop}, various general properties of $q$-Ehrhart
series and polynomials are described. In section \ref{poset}, the
special case of order polytopes and the relationship with
$P$-partitions are considered. Section \ref{empty} deals with the
special case of empty polytopes.

\thanks{Thanks to Michèle Vergne for her very pertinent comments.}

\section{$q$-Ehrhart series}

\label{qseries}

In this section, one introduces the $q$-Ehrhart series, which is a
generating series for some weighted sums over dilatations of a
polytope, and describes this series as a rational function.

This section should not be considered as completely original: similar
series have been considered in many places, including
\cite{brion_ens,brion_bourbaki}. It is therefore essentially
a brief account in our own notations of more or less classical
material, adapted to a specific context.

\medskip

Let $M$ be a lattice, let $Q$ be a lattice polytope in $M$ and let $\lambda \in M^*$ be a linear form.

One will always assume that the pair $(Q,\lambda)$
satisfies the following conditions:
\begin{description}
  \item[Positivity] For every vertex $x$ of $Q$, $\lambda(x) \geq 0$.
  \item[Genericity] For every edge $x\text{---}y$ of $Q$, $\lambda(x) \not= \lambda(y)$.
\end{description}

Let $q$ be a variable. Let us define the weighted sum over lattice points
\begin{equation}
  W_{\lambda}(Q,q) = \sum_{x \in Q} q^{\lambda(x)},
\end{equation}
and the $q$-Ehrhart series
\begin{equation}
  \ehr_{Q,\lambda}(t,q) = \sum_{n\geq 0} W_{\lambda}(n Q,q) t^n,
\end{equation}
where $n Q$ is the dilatation of $Q$ by a factor $n$.

When $q=1$, the weighted sum becomes the number of lattice points and
the $q$-Ehrhart series becomes the classical Ehrhart series.

\begin{proposition}
  \label{poles_simples}
  The $q$-Ehrhart series $\ehr_{Q,\lambda}$ is a rational function in
  $t$ and $q$. Its denominator is a product without multiplicities of
  factors $1-t q^j$ for some integers $j$ with $0 \leq j \leq
  \max_Q(\lambda)$. The factor with index $j$ can appear only if there is a
  vertex $v$ of $q$ such that $\lambda(v)=j$.
\end{proposition}

The proof will be as follows. The first step is to define a special
triangulation of $Q$, depending on $\lambda$. Then the $q$-Ehrhart series
of $Q$ is an alternating sum of $q$-Ehrhart series of simplices of the
special triangulation. The last step is to prove that the $q$-Ehrhart
series of every simplex of the chosen triangulation has the expected
properties of $\ehr_{Q,\lambda}$. This implies the proposition, as
these properties are stable by linear combinations.

\medskip

It is well-known, see for example \cite[Theorem 3.1]{beck}, that every
polytope can be triangulated using no new vertices. One will need here
a special triangulation, which depends on the linear form $\lambda$.

Remark that the Genericity condition implies that every face $F$ of
$Q$ contains a unique vertex where $\lambda$ is minimal. Let us call
it the minimal vertex of $F$.

\begin{proposition}
  There exists a unique triangulation of $Q$ with vertices
  the vertices of $Q$, such that every simplex contained in a face $F$ of $Q$
  contains the minimal vertex of $F$.

  In this triangulation, for every edge $x\text{---}y$ of every
  simplex, $\lambda(x) \not= \lambda(y)$.
\end{proposition}

\begin{proof}
  By induction on the dimension of $Q$. 

  The statement is clear if the dimension is $0$.

  Assume that the dimension is at least $1$. Let $x_0$ be the minimal
  vertex of $Q$.

  By induction, there exists a triangulation of every facet
  $F$ of $Q$ not containing $x_0$, with the stated properties.

  If two facets $F$ and $F'$ of $Q$ share a face $G$, then the
  restrictions of their triangulations give two triangulations of $G$,
  both with the stated properties. By uniqueness, they must be the
  same.

  Hence there exists a triangulation of the union of all facets of $Q$
  not containing $x_0$. One can define a triangulation of $Q$ by
  adding $x_0$ to every simplex.

  Conversely, this is the unique triangulation with the stated
  properties. Indeed, any maximal simplex of such a triangulation must
  contain the vertex $x_0$, hence the triangulation must be induced by
  triangulations of faces not containing $x_0$. By restriction, these
  triangulations must have the stated properties. Uniqueness follows.

  The second property of this triangulation follows from its inductive
  construction. Indeed, the value of $\lambda$ at $x_0$ is strictly
  less than the value of $\lambda$ at every other vertex of $Q$.
\end{proof}

By the principle of inclusion-exclusion, one can then write the
$q$-Ehrhart series of $Q$ as an alternating sum of $q$-Ehrhart series of all
simplices of the special triangulation. Let us now describe these
summands.

\begin{proposition}
  \label{serie_simplexe}
  Let $S$ be a lattice simplex, such that for every pair of distinct
  vertices $x,y$ of $S$, one has $\lambda(x) \not= \lambda(y)$. Then
  the $q$-Ehrhart series $\ehr_{Q,\lambda}$ is a rational function in
  $t$ and $q$ whose denominator is the product of $1-t q^j$ for all
  integers $j = \lambda(v)$ where $v$ is a vertex of $S$. The integers
  $j$ satisfy $0 \leq j \leq \max_S(\lambda)$.
\end{proposition}

\begin{proof}
  This is a special case of a classical result, see for instance
  \cite[Theorem 3.5]{beck}.

  Consider the cone $C$ over the simplex $\{1\} \times S$, in the
  product space $\ZZ \times M$. The generating rays of $C$ are exactly
  the vectors $(1,v)$ for vertices $v$ of $S$. According to the cited
  theorem, the poles of the generating series for $C$, which is also
  the $q$-Ehrhart series for $S$, are given by a factor $1-t
  q^{\lambda(v)}$ for every vertex $v$ of $S$. By the hypothesis, all
  these poles are distinct. By the Positivity condition, the exponents
  of $q$ are positive.
\end{proof}

From all this, one immediately obtains the statement of proposition
\ref{poles_simples}.

\begin{proposition}
  \label{min_et_max}
  The factor $1-q^{\min_Q(\lambda)}t$ and the factor
  $1-q^{\max_Q(\lambda)}t$ are always present in the denominator of
  $\ehr_{Q,\lambda}$.
\end{proposition}
\begin{proof}
  Let $m$ be $\max_Q(\lambda)$. Then for every $n\geq 0$, there is
  exactly one term $q^{n m}$ in the weighted sum $W_{\lambda}(n Q,q)$,
  corresponding to the unique maximal vertex of $n Q$.

  On the other hand, let $\ell$ be the maximal $j$ such that $1-q^j t$
  is a pole of $\ehr_{Q,\lambda}$. In the Taylor expansion of this
  fraction, the coefficient of $t^n$ is a polynomial in $q$ with
  degree at most $n \ell + k$, for some $k\geq 0$ which does not
  depend on $n$. It follows that $\ell \geq m$.

  By proposition \ref{poles_simples}, it is already known that $\ell
  \leq m$. Therefore $\ell=m$ and the result follows for the pole $1
  -q^m t$.

  The proof for the minimal case is similar.
\end{proof}

\begin{remark}
  The value of the $q$-Ehrhart series at $t=0$ is $1$, because the
  unique point in $0 Q$ is $\{0\}$.
\end{remark}

\begin{remark}
  Contrary to classical Ehrhart series, the numerator does not always
  have only positive coefficients, see example \ref{exc} below.
\end{remark}

\subsection{Examples}

\begin{example}
  \label{exa}
  Consider the polytope in $\ZZ$ with vertices $(0),(1)$ and
  the linear form $(1)$. The $q$-Ehrhart series is
  \begin{equation}
    \frac{1}{(1-t)(1-q t)} = 1 + (1+q) t + (1+q+q^2) t^2 + \cdots.
  \end{equation}
\end{example}

\begin{example}
  \label{exb}
  Consider the polytope in $\ZZ^2$ with vertices $(0,0),(1,0),(1,1)$ and
  the linear form $(1,1)$. The $q$-Ehrhart series is
  \begin{equation}
    \frac{1}{(1-t)(1-q t)(1-q^2 t)} = 1 + (1+q+q^2) t + (q^4 + q^3 + 2 q^2 + q + 1) t^2 + \dots
  \end{equation}
\end{example}

% \begin{example}
%   Consider the polytope in $\ZZ^2$ with vertices $(0,0),(1,0),(0,1)$ and
%   the linear form $(1,1)$. The Genericity condition does not
%   hold. The $q$-Ehrhart series is
%   \begin{equation}
%     \frac{t}{(1-t)(1-qt)^2},
%   \end{equation}
%   with a double pole.
% \end{example}

\begin{example}
  \label{exc}
  Consider the polytope in $\ZZ^2$ with vertices $(0,0),(1,0)$,
  $(1,1),(2,1)$ and the linear form $(1,1)$. The $q$-Ehrhart series is
  \begin{equation}
    \frac{1- q^3 t^2}{(1 - t)(1-q t)(1-q^2 t)(1- q^3 t)} = 1 + (1+q+q^2+q^3) t + \dots
  \end{equation}
  Note that its numerator has a negative coefficient.
\end{example}

\begin{example}
  \label{exd}
  Consider the polytope in $\ZZ^2$ with vertices
  $(0,0),(1,0),(1,1),(0,3)$ and the linear form $(1,1)$. The $q$-Ehrhart
  series is
  \begin{equation}
    \frac{ 1 + (q^2+q)t - (q^4+q^3+q^2)t^2}{(1-t)(1-q t)(1-q^2 t)(1-q^3 t)}.
  \end{equation}
\end{example}

% \begin{example}
%   Let $a,b,c,d \in \ZZ_{> 0}$ pairwise distinct. Consider the
%   $3$-dimensional polytope $Q_{a,b,c,d}$ in $\ZZ^4$ defined by
%   \begin{equation}
%     \frac{x}{a}+\frac{y}{b}+\frac{z}{c}+\frac{w}{d} = 1,
%   \end{equation}
%   and $x,y,z,w \geq 0$. Let us choose the linear form $(1,1,1,1)$. The
%   $q$-Ehrhart series is
%   \begin{equation}
%     \frac{1}{(1-q^a t)(1-q^b t)(1-q^c t)(1-q^d t)}.
%   \end{equation}
% %  This has simple poles if and only if $a,b,c,d$ are pairwise distinct.
% \end{example}

\section{$q$-Ehrhart polynomial}

\label{qpoly}

In this section, one proves the existence of the $q$-Ehrhart
polynomial, and obtains a $q$-analog of Ehrhart reciprocity.

\medskip

Let $[n]_q$ be the $q$-integer
\begin{equation*}
  [n]_q = \frac{q^n-1}{q-1}.
\end{equation*}

Let $Q$ be a lattice polytope and $\lambda$ be a linear form that
satisfy the Positivity and Genericity conditions.

Let us write $m$ for $\max_Q(\lambda)$.

\begin{theorem}
  \label{ehrhart_poly}
  There exists a polynomial $L_{Q,\lambda} \in \QQ(q)[x]$ such that
  \begin{equation}
    \forall\, n \in \ZZ_{\geq 0} \quad L_{Q,\lambda}([n]_q) = W_\lambda (n Q,q).
  \end{equation}
  The degree of $L_{Q,\lambda}$ is $m$. The coefficients of
  $L_{Q,\lambda}$ have poles only at roots of unity of order less than
  $m$.
\end{theorem}

\begin{proof}
  Consider the $q$-Ehrhart series $\ehr_{Q,\lambda}$. By Proposition
  \ref{poles_simples}, it can be written as a sum
  \begin{equation}
    \sum_{j=0}^{m} c_j \frac{1}{1-q^j t},
  \end{equation}
  for some coefficients $c_j$ in $\QQ(q)$. Expanding one of the simple
  fractions, one gets
  \begin{equation*}
    \frac{1}{1-q^j t} = \sum_{n\geq 0 } q^{n j} t^n.
  \end{equation*}

  Because
  \begin{equation}
    \label{coeur_eval}
    (1+q x-x)\mid_{x=[n]_q} = q^n,
  \end{equation}
  the value of the polynomial $(1 + q x - x)^j$ at the $q$-integer
  $[n]_q$ is given by $q^{n j}$.
  
  Define the polynomial $L_{Q,\lambda}(x)$ by
  \begin{equation*}
    \sum_{j=0}^{m} c_j (1 + q x - x)^j.
  \end{equation*}
  It follows that the value $L_{Q,\lambda}([n]_q)$ is exactly the
  coefficient of $t^n$ in the $q$-Ehrhart series $\ehr_{Q,\lambda}$. This
  is the expected property.

  The statement about the degree of $L_{Q,\lambda}(x)$ is clear from
  the previous formula and proposition \ref{min_et_max}.

  The polynomial $L_{Q,\lambda}(x)$ can therefore be recovered by
  interpolation at the $q$-integers between $[0]_q $ and $[m]_q$. The
  stated property of poles of its coefficients follows.
\end{proof}

\begin{remark}
  Contrary to the case of classical Ehrhart polynomials, whose degree
  is bounded by the ambient dimension, the degree here is the maximal
  value of the linear form on the polytope, and can be arbitrary large
  in any fixed dimension.
\end{remark}

\begin{remark}
  Obviously, letting $q=1$ in the $q$-Ehrhart polynomial recovers the
  classical Ehrhart polynomial.
\end{remark}

\begin{example}
  Consider the four polytopes of examples \ref{exa},
  \ref{exb},\ref{exc} and \ref{exd}. Their $q$-Ehrhart polynomials are
  \begin{align*}
    & q x + 1,\\
    & \frac{(q x + 1) (q^2 x + q + 1)}{q + 1},\\
    & \frac{(q x + 1)^2 (q^2 x - q x + q + 1)}{q + 1},\\
    & (q x + 1) (q (q - 1) x^2 + 2 q x + 1).
  \end{align*}
  The reader can check the values at $x = 0,1$ and the reduction to
  the classical Ehrhart polynomial at $q=1$. All four examples being
  empty lattice polytopes, the values at $[-1]_q$ vanish.
\end{example}

\subsection{$q$-Ehrhart reciprocity}

If $Q$ is a polytope, let us denote by $\int(Q)$ the interior of $Q$.

Let
\begin{equation}
  W_{\lambda}(\int(n Q),q) = \sum_{x \in \int(n Q)} q^{\lambda(x)}
\end{equation}
be the weighted sum over interior lattice points in $n Q$.

Let $L_{Q,\lambda}$ be the $q$-Ehrhart polynomial of $(Q,\lambda)$.

The following theorem is a $q$-analogue of Ehrhart reciprocity.

\begin{theorem}
  \label{ehrhart_reciprocity}
  For every integer $n \in \ZZ_{>0}$, one has
  \begin{equation}
    L_{Q,\lambda}([-n]_q) = (-1)^{d} W_{\lambda}(\int(n Q),1/q),
  \end{equation}
  where $d$ is the dimension of $Q$.
\end{theorem}

\begin{proof}  
  By Stanley's reciprocity theorem for rational cones \cite[Theorem
  4.3]{beck}, applied to the cone over the polytope $Q$ and with
  variables specialised to $t$ and appropriate powers of $q$, one
  obtains
  \begin{equation}
    \label{stanley_reciprocity}
    \ehr_{Q,\lambda}(1/t,1/q) = (-1)^{d+1} \ehr_{\int(Q),\lambda}(t,q),
  \end{equation}
  where
  \begin{equation}
    \ehr_{\int(Q),\lambda}(t,q) = \sum_{n \geq 1} W_{\lambda}(\int(n Q),q) t^n.
  \end{equation}

  By definition of the $q$-Ehrhart series, one has
  \begin{equation*}
    \ehr_{Q,\lambda}(1/t,1/q) = \sum_{n \leq 0}  L_{Q,\lambda}([-n]_{1/q}) t^n.
  \end{equation*}
  By lemma \ref{lemme_zero} below, this is the same as
  \begin{equation}
    - \sum_{n \geq 1} L_{Q,\lambda}([-n]_{1/q}) t^n.
  \end{equation}

  From this, one deduces that
  \begin{equation}
    \sum_{n \geq 1} W_{\lambda}(\int(n Q),q) t^n =
    (-1)^d \sum_{n \geq 1} L_{Q,\lambda}([-n]_{1/q}) t^n,
  \end{equation}
  which is equivalent to the statement of the proposition.
\end{proof}

\begin{lemma}
  \label{lemme_zero}
  Let $P$ be a polynomial in $x$ with coefficients in $\QQ(q)$.  Then
  \begin{equation*}
    F^{+} = \sum_{n\geq 0} P([n]_q)t^n \quad\text{and}\quad F^{-} = \sum_{n<0} P([n]_q)t^n
  \end{equation*}
  are rational functions in $t,q$ and $F^{+}+F^{-}=0$.
\end{lemma}
\begin{proof}
  Every such polynomial can be written as a finite sum
  \begin{equation*}
    \sum_{j} c_j (1+q x-x)^j,
  \end{equation*}
  for some coefficients $c_j$ in $\QQ(x)$.

  By linearity, it in enough to prove the lemma for the polynomial
  $(1+q x-x)^j$. In this case, using \eqref{coeur_eval}, one finds that
  $F^{+} = 1/(1-q^j t)$ and $F^{-}=q^{-j}t^{-1}/(1-q^{-j}t^{-1})$. The statement
  is readily checked.
\end{proof}

\section{Other properties}

\label{otherprop}

In this section, various general properties of the $q$-Ehrhart series
and the $q$-Ehrhart polynomials are described.

\subsection{Shifting the linear form}

Let $Q$ be a polytope. Let $s(Q)$ be the image of $Q$ by a translation
by a vector $v$ such that $\lambda(v)=N \geq 0$. The Positivity and
Genericity conditions still hold for $s(Q)$.

At the level of $q$-Ehrhart series, it is immediate to see that
\begin{equation}
  \ehr_{s(Q),\lambda}(t,q) = \ehr_{Q,\lambda} (q^N t,q).
\end{equation}
and that
\begin{equation}
  W_{\lambda}(n s(Q)) = q^{N n} W_{\lambda}(n Q).
\end{equation}

Using \eqref{coeur_eval}, one obtains that, at the level of $q$-Ehrhart polynomial, 
\begin{equation}
  L_{s(Q),\lambda} = (1+q x-x)^N L_{Q,\lambda}.
\end{equation}

\begin{remark}
  By using this kind of shift, one can always
  assume that $0 \in Q$.
\end{remark}

\subsection{Reversal of polytopes}

\label{reversal}

One defines here a duality on polytopes, depending on $\lambda$.

By the Genericity condition, there exists a unique vertex $v_{\max}
\in Q$ where $\lambda$ is maximal. Let us define a polytope $
\overline{Q}$ as $v_{\max} - Q$. It is therefore the image of $Q$ by
an integer affine map which exchanges $0$ and $v_{\max}$, hence $Q$
and $\overline{Q}$ are equivalent as lattice polytopes.

The Positivity and Genericity conditions still hold for
$\overline{Q}$.

In general, the pairs $(Q,\lambda)$ and $(\overline{Q},\lambda)$ are
not equivalent under the action of the integral affine group, but some
pairs $(Q,\lambda)$ can be isomorphic to their dual for this
duality. A necessary condition is that $0$ is in $Q$.

\begin{proposition}
  \label{prop_reversal}
  The effect of this duality on $q$-Ehrhart series is given by
  \begin{equation}
    \ehr_{\overline{Q},\lambda} = \ehr_{Q,\lambda} (t q^m, 1/q),
  \end{equation}
  where $m = \lambda(v_{\max})$ is the maximal value of $\lambda$ on $Q$.
\end{proposition}

\begin{proof}
  One can see that $\overline{n Q}$ is just $n \overline{Q}$ for every
  $n \geq 0$. Therefore every point of weight $q^j$ in some $n Q$
  corresponds to a point of weight $q^{n m -j}$ in $n
  \overline{Q}$. This implies the statement.
\end{proof}

\begin{remark}
  The $q$-Ehrhart series of a polytope and its reversal are usually
  distinct, unless the polytope is self-dual. But they give the same
  classical Ehrhart series when $q=1$.
\end{remark}

\subsection{Many different pyramids}

\label{pyramides}

Let $Q$ be a lattice polytope in the lattice $M$.

Define a new polytope $\pyr(Q)$ in the lattice $\ZZ \times M$ as the
pyramid with apex $(1,0)$ based on $(0,Q)$. This is the convex hull of
the polytope $Q$ and a new vertex placed in a shifted parallel space.

Let us choose an integer $m \geq 0$ such that $m$ is not among the
values of $\lambda$ on $Q$. For example, one can always choose
$\max_Q(\lambda)+1$.

Let us extend the linear form $\lambda$ to a linear form $m \oplus\lambda$
on the lattice $\ZZ \times M$, whose value on a vector $(k,v)$ in $\ZZ
\times M$ is $k m + \lambda(v)$.

The Positivity and Genericity conditions still hold for $\pyr(Q)$ with
respect to $m \oplus\lambda$.

\begin{proposition}
  \label{pour_pyr}
  The $q$-Ehrhart series of $(\pyr(Q),m \oplus\lambda)$ is given by
  \begin{equation}
    \ehr_{\pyr(Q),m \oplus\lambda} = \ehr_{Q,\lambda} / ( 1- q^{m} t).
  \end{equation}
\end{proposition}

\begin{proof}
  Let us compute
    \begin{equation}
    \ehr_{\pyr(Q),m \oplus\lambda} = \sum_{n \geq 0} W_{m \oplus\lambda}(n \pyr(Q),q) t^n.
  \end{equation}
  By the definitions of $\pyr(Q)$ and $m \oplus\lambda$, this is
  \begin{equation}
    \sum_{n \geq 0} \sum_{i=0}^{n} q^{m i} W_{\lambda}((n-i) Q,q) t^n =
    \sum_{n \geq i \geq 0} q^{m i} t^i W_{\lambda}((n-i) Q,q) t^{n-i},
  \end{equation}
  which can be rewritten as the expected result.
\end{proof}

\subsection{Periodicity of values at cyclotomic $q$}

\label{period_cyclotomic}

Let $N$ be an integer such that $N > \max_Q(\lambda)$, and let
$\xi$ be a primitive root of unity of order $N$.

By theorem \ref{ehrhart_poly}, one can let $q=\xi$ in the
$q$-Ehrhart polynomial $L_{Q,\lambda}$.

\begin{proposition}
  The sequence of values $L_{Q,\lambda}([n]_q)|_{q=\xi}$ for $n \in
  \ZZ$ is periodic of period $N$.
\end{proposition}
\begin{proof}
  Indeed, the sequence $[n]_q$ itself is periodic of period $N$.
\end{proof}

Assume now that there is no lattice point in $\int(n Q)$ for some
integer $n$. By $q$-Ehrhart reciprocity, one has
$L_{Q,\lambda}([-n]_q)=0$. By the previous proposition, one deduces
that
\begin{equation}
  L_{Q,\lambda}([-n+k N]_q)|_{q=\xi} = 0,
\end{equation}
for all $k \in \ZZ$. This means that the cyclotomic polynomial
$\Phi_N$ divides the value $L_{Q,\lambda}([-n+k N]_q)$ for all $k \in
\ZZ$. 

This construction provides many cyclotomic factors in the values of
some $q$-Ehrhart polynomials.

\section{Posets and $P$-partitions}

\label{poset}

There exists a well-known theory of $P$-partitions, due to R. Stanley
\cite{stanley_ordered}, which describes decreasing colourings of
partially ordered sets (see also \cite{feray_reiner}). Part of this
theory, namely its restriction to natural labellings, coincides exactly
with a special case of the theory developed here, namely its
application to the order polytope of the opposite of a poset. The
theory of $P$-partitions does not include any analog of Ehrhart
polynomials.

This section describes this common special case, and some specific
properties of $q$-Ehrhart series and $q$-Ehrhart polynomials for
posets.

\medskip

Let $P$ be a finite poset. The order polytope $Q_P$ of the poset $P$
is a lattice polytope in $\ZZ^P$ (with coordinates $z_x$ for $x \in
P$), defined by the inequalities
\begin{align*}
  0 \leq z_x \leq 1 &\quad \forall\, x \in P,\\
  z_x \leq z_y &\quad \text{if}\quad x \leq y \in P.
\end{align*}

The polytope $Q_P$ has vertices in $\ZZ^{\{0,1\}}$ and no interior
lattice point \cite{stanley_2p}.

Points in the dilated polytope $n Q_P$ correspond to increasing colourings of the
elements of $P$ by the integers in $\{0,\dots,n\}$.

In this section, the linear form $\lambda$ will always be given by the
sum of coordinates. The Positivity condition is clearly satisfied by
$Q_P$ and this linear form. One can also check that the Genericity
condition holds, by using the known description of the vertices and
edges of the order polytopes \cite{stanley_2p}. The minimal and
maximal values of $\lambda$ on $Q_P$ are $0$ and the cardinality of
$P$.

For short, one will denote $\ehr_P$ and $L_P$ for the $q$-Ehrhart
series and polynomial of $Q_P$.

According to \cite[\S 8]{stanley_ordered}, the $q$-Ehrhart series
$\ehr_P$ can be written
\begin{equation}
  \ehr_P = \frac{W_P}{(1- t)(1-q t)\dots(1-q^{\#P} t)},
\end{equation}
where $W_P$ is a polynomial in $q$ and $t$ with nonnegative integer
coefficients. This polynomial has a known combinatorial
interpretation, using descents and major indices, as a sum over all
linear extensions of the poset $P$.

From the general existence of the $q$-Ehrhart polynomial for polytopes
(theorem \ref{ehrhart_poly}), one deduces 

\begin{proposition}
  There exists a polynomial $L_P$, of degree $\#P$, such that
  $L_P([n]_q)$ is the weighted sum over increasing colourings of $P$ by
  $\{0,\dots,n\}$, where the weight is $q$ to the power the sum of
  colours.
\end{proposition}

From the $q$-Ehrhart reciprocity (theorem \ref{ehrhart_reciprocity}), one obtains 

\begin{proposition}
  For every integer $n\geq $, the polynomial $(-1)^{\# P} L_P([-n]_q)$
  is the weighted sum over strictly increasing colourings of $P$ by
  $\{1,\dots,n-1\}$, where the weight is $q$ to the power the sum of
  colours.
\end{proposition}

One can find in \cite[Prop. 10.4]{stanley_ordered} a reciprocity
formula for the $q$-Ehrhart series, closely related to the previous
proposition.

\begin{figure}[h!]
  \centering
  \includegraphics[height=2cm]{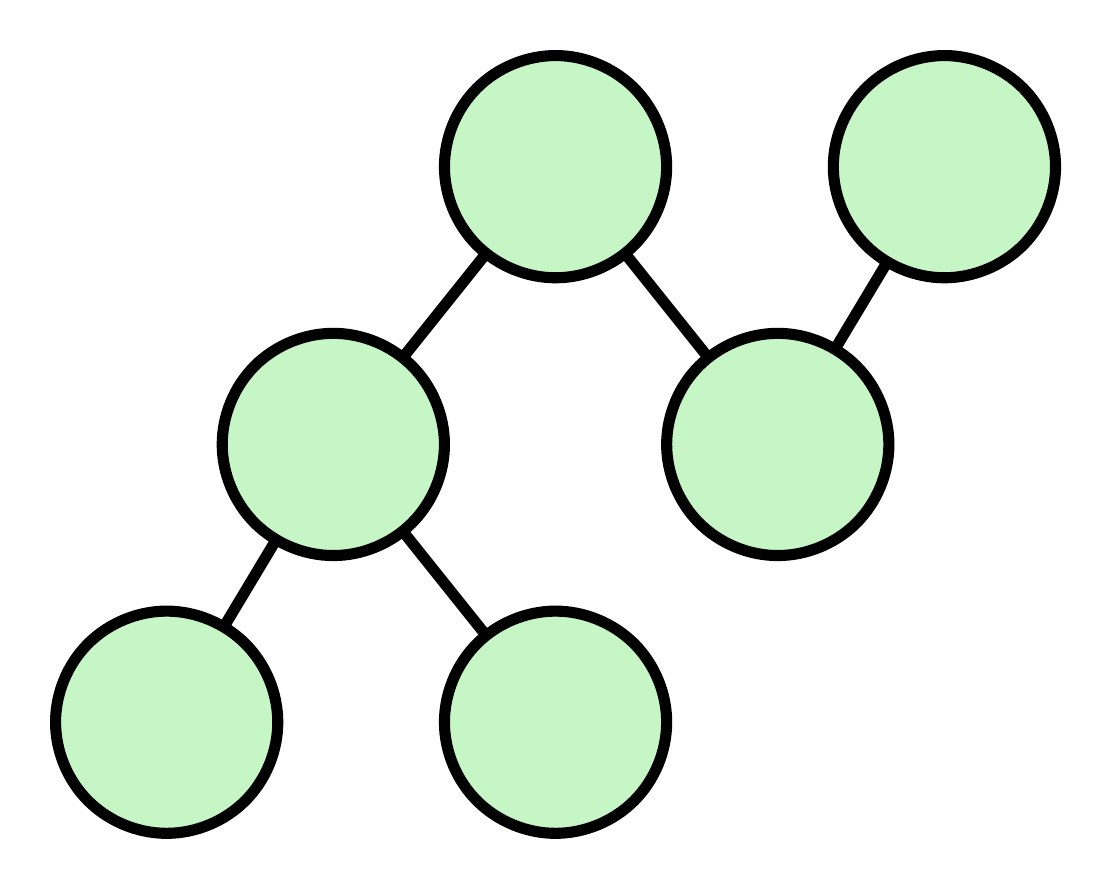}
  \caption{A poset $P$ on $6$ vertices, minima at the bottom}
  \label{fig:exemple_poset}
\end{figure}

\begin{figure}[h!]
  \centering
  \includegraphics[height=2.5cm]{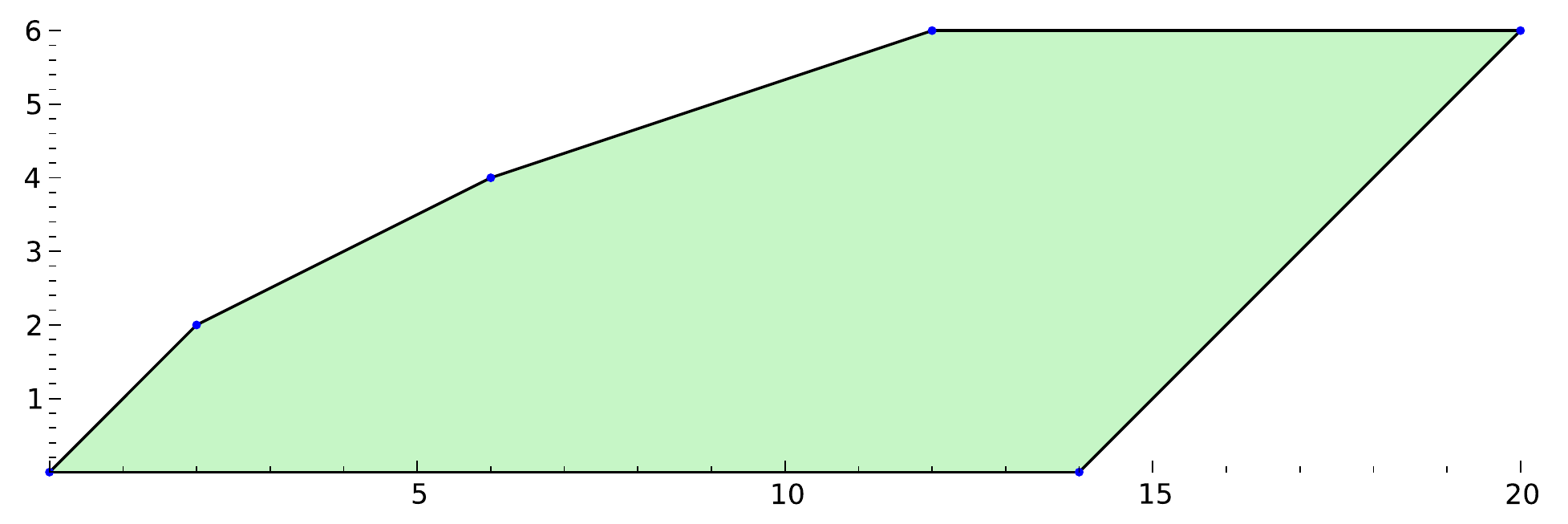}
  \caption{Newton polytope of the numerator of $L_P$}
  \label{fig:newton_poset}
\end{figure}

Based on experimental observations, one proposes the following
conjecture, illustrated in figure \ref{fig:newton_poset}.

\begin{conjecture}
  The Newton polytope of the numerator of the $q$-Ehrhart polynomial
  $L_P$ has the following shape. It has an horizontal top edge, an
  horizontal bottom edge and a diagonal right edge. Every element $x$
  of the poset $P$ gives rise to a segment on the left border with
  inverse slope given by the length of the maximal chain of elements
  larger than $x$.
\end{conjecture}

For the trivial poset with just one element, the Ehrhart polynomial is $1+q x$.

\begin{example}
  Let $P$ be the partial order on the set $\{a,b,c,d\}$, where $a$ is
  smaller than $b,c,d$. Then the Ehrhart polynomial $L_P$ is $q x + 1$ times 
  \begin{equation} % ok
    \frac{(q^2 x + q + 1)(\Phi_3 \Phi_4 + q (2 q^4 + 4 q^2 + q + 2) x + q^2 (q^4 - q^3 + 3 q^2 - q + 1) x^2)}{\Phi_2 \Phi_3 \Phi_4},
  \end{equation}
  where $\Phi_i$ is the cyclotomic polynomial of order $i$ in the variable $q$.

  For the opposite poset, one finds instead
  \begin{equation} % ok
    \frac{(q x + 1)(q^2 x + q + 1)(\Phi_3 q^4 x^2 + (2 q^4 + 2 q^3  + 3 q + 2 ) q^2 x + \Phi_4 \Phi_3)}{\Phi_2 \Phi_3 \Phi_4}.
  \end{equation}
\end{example}

\subsection{Value at infinity}

Given a poset $P$, one can evaluate the $q$-Ehrhart polynomial $L_P$
at the limit (as a formal power series in $q$) of $q$-integers $[n]_q$
when $n$ becomes infinity, namely at $1/(1-q)$. This gives a rational
function in $q$, which corresponds to the weighted sum of all
increasing colourings of $P$.

For example, for the partial order on $\{a,b,c,d\}$ with $a\leq
b,c,d$, one gets
\begin{equation}
  \frac{1}{(q - 1)^4 \Phi_2 \Phi_4},
\end{equation}
and for the opposite poset, one obtains
\begin{equation}
  \frac{q^4 - q^3 + 3 q^2 - q + 1}{(q - 1)^4 \Phi_2 \Phi_3 \Phi_4}.
\end{equation}

Let us compare this value to a limit at $t=1$ of the $q$-Ehrhart
series.

\begin{proposition}
  \label{n_infini}
  The value of $L_p$ at $x = 1/(1-q)$ is also the value at $t=1$ of
  the product $(1-t) \ehr_P$.
\end{proposition}
\begin{proof}
  Indeed, the value of $L_p$ at $x = 1/(1-q)$ is the limit of the
  weighted sums over the dilated order polytopes of the poset
  $P$. This is just a weighted sum over the cone defined by the poset
  $P$.
  
  On the other hand, the series $\ehr_P$ has a simple pole at $t=1$,
  hence the product $(1-t) \ehr_P$ has a well-defined value at
  $t=1$. The coefficients of the series $(1-t) \ehr_P$ are the
  differences $L_p([n]_q)-L_P([n-1]_q)$. Their sum is also the
  weighted generating series of the cone associated with the poset
  $P$.
\end{proof}

\subsection{Volume}

\begin{remark}
  \label{dual_poset}
  Let us note that the order polytope $Q_{\overline{P}}$ for the
  opposite $\overline{P}$ of a poset $P$ is the reversal of the order
  polytope $Q_p$, as defined in section \ref{reversal}.
\end{remark}

Therefore, by proposition \eqref{prop_reversal}, one has
\begin{equation}
  \label{Pbar_P}
  \ehr_{\overline{P}} = \ehr_P (t q^{\# P}, 1/q).
\end{equation}

Let
\begin{equation}
  \qbinom{n}{m}
\end{equation}
denote the $q$-binomial coefficients.

\begin{lemma}
  \label{q_binomial}
  Let $d$ be a nonnegative integer. Then
  \begin{equation}
    \frac{1}{\prod_{j=0}^{d} 1-q^j t} = \sum_{n\geq 0} \qbinom{d + n}{n} t^n.
  \end{equation}
  The coefficient of $t^n$ is the value of the polynomial
  \begin{equation}
    \prod_{j=1}^{d}\frac{[j]_q+q^j x}{[j]_q}
  \end{equation}
  at the $q$-integer $[n]_q$.
\end{lemma}
\begin{proof}
  The first equation is classical and can be proved by an easy
  induction on $d$, using the definition of the $q$-binomial
  coefficients. The second statement is then clear.
\end{proof}

\begin{lemma}
  \label{shift_interpol}
  Let $f$ be a polynomial in $\QQ(q)[x]$ of degree $d$. Let $g$ be the
  polynomial $f(1+q x)$. Then $g([n]_q)= f([n+1]_q)$ for every integer
  $n$. The leading coefficient of $g$ is $q^d$ times the leading
  coefficient of $f$.
\end{lemma}
\begin{proof}
  Obvious.
\end{proof}

Let 
\begin{equation*}
  [n]!_q = [1]_q [2]_q \cdots [n]_q
\end{equation*}
be the $q$-factorial of $n$.

\begin{definition}
  The $q$-volume of a poset $P$ is the leading coefficient of the
  $q$-Ehrhart polynomial $L_{P}$ times the $q$-factorial of $\#P$.
\end{definition}

For example, for the partial order on $\{a,b,c,d\}$ with $a\leq
b,c,d$, the $q$-volume is
\begin{equation}
  q^5 (q + 1) (q^4 - q^3 + 3 q^2 - q + 1),
\end{equation}
and for opposite poset, it is given by
\begin{equation}
  q^7 (q + 1) (q^2 + q + 1).
\end{equation}

\begin{proposition}
  The $q$-volume of $P$ is equal to $q^{\binom{\#P+1}{2}}$ times
  the value at $t=1$ and $q=1/q$ of the numerator of the $q$-Ehrhart
  series of the opposite poset $\overline{P}$.
\end{proposition}
\begin{proof}
  Let us write
  \begin{equation}
    \label{ansatz}
    \ehr_P = \frac{\sum_{k=0}^{\# P} h_k t^k}{\prod_{j=0}^{\#P} 1-q^j t},
  \end{equation}
  for some coefficients $h_k$ in $\QQ[q]$.

  According to \eqref{Pbar_P}, the value at $t=1$ of the numerator of
  $\ehr_{\overline{P}}(q=1/q)$ is also the value at $t=1$ of the
  numerator of $\ehr_{P}(t q^{-\# P})$. This is given by
  \begin{equation*}
    \sum_{k=0}^{\# P} h_k q^{-k \# P}.
  \end{equation*}

  By lemma \ref{q_binomial} and lemma \ref{shift_interpol} applied to
  \eqref{ansatz}, the leading coefficient of the polynomial $L_P$ is
  given by
  \begin{equation}
    \left(\sum_{k=0}^{\# P} h_k q^{-k \# P}\right) \prod_{j=1}^{\# P}\frac{q^j }{[j]_q}
    = \left(\sum_{k=0}^{\# P} h_k q^{-k \# P}\right) \frac{q^{\binom{\# P+1}{2}}}{[\# P]!_q}.
  \end{equation}
  Comparing with the previous formula and using the definition of the
  $q$-volume concludes the proof.
\end{proof}

\subsection{Pyramids for posets}

As a special case of the general pyramid construction for polytopes
described in section \S \ref{pyramides}, one obtains the following
results.

\begin{proposition}
  Let $P$ be a poset. Consider the poset $P^-$ with one minimal element
  added. Then the $q$-Ehrhart series are related by
  \begin{equation}
    \ehr_{P^-} = \ehr_{P} / ( 1- q^{1+\# P} t).
  \end{equation}
\end{proposition}
\begin{proof}
  Indeed the order polytope of $P^-$ is a pyramid over the product
  $\{0\} \times Q_P$, with one more vertex where every coordinate is
  $1$. The sum-of-coordinates linear form takes the value $ 1+\# P$ on
  this vertex. This pair (polytope, linear form) is equivalent as a
  pair to $(\pyr(Q_P),m\oplus\lambda)$ with $m = 1+\# P$ and $\lambda$
  the sum-of-coordinates linear form on $Q_P$. The result then follows
  from proposition \ref{pour_pyr}.
\end{proof}

\begin{proposition}
  Consider the poset $P^+$ with one maximal element added. Then
  the $q$-Ehrhart series are related by
  \begin{equation}
    \ehr_{P^+}(t,q) = \ehr_{P} (q t,q) / ( 1- t).
  \end{equation}
\end{proposition}
\begin{proof}
  This can be deduced from the previous proposition for the opposite
  poset, remark \ref{dual_poset} and proposition \ref{prop_reversal}.
\end{proof}

\subsection{Vanishing at small negative $q$-integers}

Let $P$ be a poset, and let $\ell$ be the length of the longest
increasing chain of $P$. 

Because there are no strictly increasing colourings of $P$ by integers
${1,\dots,n-1}$ if $n \leq \ell$, the $q$-Ehrhart polynomial $L_P$
vanishes at $[-n]_q$ for $1 \leq n \leq \ell$.

This implies that the $q$-Ehrhart polynomial $L_P$ is divisible by
$[n]+q^n x$ for every $1\leq n\leq \ell$.

By the remarks of \S \ref{period_cyclotomic}, this gives many
cyclotomic factors in the values of $L_P$ at $q$-integers.

\subsection{$q$-Ehrhart polynomials of minuscule posets}

For some posets, there are simple product formulas for the weighted
sums over increasing colourings. One can deduce from them product
formulas for the $q$-Ehrhart polynomials of these posets.

For example, consider the poset $P_{m,n} = A_m \times A_n$, where
$A_m$ is the total order of size $m$.

A famous result of MacMahon, usually described using plane partitions
inside a box of size $m \times n \times k$, states that the weighted
sum of decreasing colourings of the poset $P_{m,n}$ by integers in
$\{0,\dots, k\}$ is given by
\begin{equation*}
  \prod_{i=1}^{m} \prod_{j=1}^{n} \frac{[i+j-1+k]_q}{[i+j-1]_q}.
\end{equation*}

From this, one can easily deduce that
\begin{proposition}
  The $q$-Ehrhart polynomial of $P_{m,n}$ is
  \begin{equation}
    \prod_{i=1}^{m} \prod_{j=1}^{n} \frac{[i+j-1]_q+ x q^{i+j-1}}{[i+j-1]_q}.
  \end{equation}
\end{proposition}

More generally, this can be applied to the similar known formulas for
other minuscule posets \cite{proctor,stembridge}, to obtain formulas for the
$q$-Ehrhart polynomial as a product of linear polynomials.

\section{Empty polytopes}

\label{empty}

An \textbf{empty lattice polytope} is a lattice polytope $Q$ such that
there is no lattice point in the interior $\int(Q)$ of $Q$.

This section considers the $q$-Ehrhart theory in the special case of
empty polytopes. For these polytopes, one can define a special
evaluation of the $q$-Ehrhart polynomial, which seems to have
interesting properties.

\subsection{Special value for empty polytopes}

\label{special_value}

Let $Q$ be an empty lattice polytope. By $q$-Ehrhart reciprocity
(theorem \ref{ehrhart_reciprocity}), the $q$-Ehrhart polynomial
$L_{Q,\lambda}$ vanishes at $[-1]_q=-1/q$, hence is divisible by $1+q
x$.

Let us define the \textbf{special value} of $Q$ by
\begin{equation}
  \sv_{Q,\lambda} =  {\left( \frac{L_{Q,\lambda}(x)}{1+q x} \right)}\bigg{|}_{x=-1/q}
\end{equation}

By construction, the special value is a fraction in $\QQ(q)$, with
possibly poles at $0$ and roots of unity.

The special value has the following property, which is not obviously
true, because there exist polytopes with multiplicities in the
denominator of their $q$-Ehrhart polynomials.

\begin{proposition}
  \label{mystery_sv}
  The poles of the special value $\sv_{Q,\lambda}$ at roots of unity
  are simple.
\end{proposition}

\begin{proof}
  Let $y = 1 + q x$ and let $L(y)$ be the polynomial
  $L_{Q,\lambda}((y-1)/q)$. One can write
  \begin{equation*}
    L(y) = \sum_{i=1}^{d} c_i y^i,
  \end{equation*}
  for some integer $d$. The proposition is then equivalent to the
  statement that $c_1$ has only simple poles at root of unity.

  The vector space over $\QQ(q)$ spanned by $y^i$ for $i=1,\dots,d$
  has another basis, given by polynomials
  \begin{equation*}
    P_i(y) = \frac{y}{[i]_q} \frac{\prod_{1 \leq j\not=i} y-[j]_q}{\prod_{1 \leq j \not=i}[i]_q-[j]_q},
  \end{equation*}
  for $i=1,\dots,d$.

  The coefficients $c_i$ are the coefficients of the polynomial $L(y)$
  in the basis $(y^i)_{1\leq i\leq d}$. The values of $L(y)$ at the $q$-integers
  $[i]_q$ are polynomials in $q$ and give the coefficients of $L(y)$
  in the basis $(P_i)_{1\leq i\leq d}$.

  The change of basis matrix from the basis $(P_i)_{1\leq i\leq d}$ to the basis
  $(y^i)_{1\leq i\leq d}$ is given by the expansion of the polynomials $P_i$ in
  powers of $y$. In particular, the coefficient $c_1$ is computed
  using only the values of $L(y)$ at $q$-integers $[i]_q$ and the
  coefficients of $y$ in the polynomials $P_i$, which are given by
  \begin{equation}
    (-1)^{d-1} \frac{1}{[i]_q}\frac{\prod_{1\leq j\not=i} [j]_q}{\prod_{1\leq j \not=i}[i]_q-[j]_q}.
  \end{equation}
  This expression can be rewritten using the $q$-binomials (up to sign
  and a power of $q$) as
  \begin{equation}
     \frac{1}{[i]_q}\qbinom{d}{i},
  \end{equation}
  and has therefore only simple poles at roots of unity. The statement follows.
\end{proof}

% One can also consider the corresponding classical statement, when $q=1$.

One may wonder whether this special value has any geometric meaning.

For the examples \ref{exa}, \ref{exb}, \ref{exc} and \ref{exd}, the
special values are $1$, $1/(1+q)$, $0$ and $-1/q$.

As the polytope $Q_P$ associated with a poset $P$ is empty, one can
define the special value $\sv_P$. In the companion paper
\cite{serieqx}, it is proved using different methods that, for every
poset $P$ which is a rooted tree (with the root as maximum), the
special value $\sv_P$ has only simple poles at root of unity.

% Example: for the zigzag poset, the special value seems to have the shape
% \begin{equation}
%   \frac{q^*}{[n]_q \begin{bmatrix}n-1 \\ (n-1)/2\end{bmatrix}_q}
% \end{equation}
% with some power of $q$ in the numerator, namely $q^0,q^0,q^1,q^1,q^3,q^3,?$

\subsection{Hahn operator}

One defines the $\QQ(q)$-linear operator $\Delta$ by
\begin{equation}
  \Delta(f) = \frac{f(1+q x)-f(x)}{1+ q x -x},
\end{equation}
acting on polynomials in $x$.

This is a $q$-analog of the derivative, which has been introduced by
Hahn in \cite{hahn}. The kernel of $\Delta$ is the space of constant
polynomials. The restriction of $\Delta$ to the space of multiples of
$1+ q x$ is an isomorphism with $\QQ(q)[x]$.

One can translate this action of $\Delta$ on polynomials into an
action on the values at $[n]_q$. Let $f_n$ be the value
$f([n]_q)$. Then
\begin{equation}
  \label{delta_val}
  \Delta(f)([n]_q) = \frac{f_{n+1}-f_{n}}{q^n}.
\end{equation}

% Then by introducing the generating series
% \begin{equation*}
%   F = \sum_{n \geq 0} f_n t^n,
% \end{equation*}
% this corresponds to an action of $\Delta$ on formal power series in $t$:
% \begin{equation}
%   \Delta F = \frac{q}{t} \big{(} F(t/q)-F(0) \big{)} - F(t/q).
% \end{equation}

% \begin{lemma}
%   One has
%   \begin{equation}
%     \Delta\big{(}\frac{1}{1-q^a t}\big{)} = \frac{q^a -1}{1-q^{a-1} t}.
%   \end{equation}
% \end{lemma}

For a lattice polytope $Q$ in the lattice $M$, define a new polytope
$\bop(Q)$ in the lattice $\ZZ \times M$ as the convex hull of the
vertex $(0,0)$ and the product $\{1\} \times Q$. As a polytope, this
is just the pyramid over $Q$. The linear form $\lambda$ on $M$ is
extended to a linear form $\bop(\lambda)$ defined by 
\begin{equation*}
  \bop(\lambda)(k,v)=k+\lambda(v).
\end{equation*}

The positivity and genericity conditions clearly hold for
$(\bop(Q),\bop(\lambda))$.

\begin{proposition}
  \label{delta_bop}
  One has
  \begin{equation}
    \Delta(L_{\bop(Q),\bop(\lambda)}) = q L_{Q,\lambda} (1+q x).
  \end{equation}
\end{proposition}
\begin{proof}
  By \eqref{delta_val}, the value $q^n
  \Delta(L_{\bop(Q),\bop(\lambda)})([n]_q)$ is the weighted sum over
  $(n+1)\bop(Q)$ minus the weighted sum over $n\bop(Q)$. By the
  definition of $\bop(Q)$, this is nothing else than $q^{n+1} $ times
  the weighted sum over $(n+1) Q$, namely
  $L_{Q,\lambda}([n+1]_q)$. The result follows.
\end{proof}

\subsection{Umbral equalities}

Recall that Carlitz $q$-Bernoulli numbers (introduced in \cite{carlitz}) are rational functions in $q$ defined by $\beta_0=1$ and
\begin{equation}
  \label{recu_beta}
  q(q \beta +1)^n-\beta_n=
  \begin{cases}
    1 \text{ if }n=1,\\
    0 \text{ if }n>1,
   \end{cases}
\end{equation}
where by convention one replaces $\beta^k$ by $\beta_k$ after
expansion of the binomial power.

The Carlitz $q$-Bernoulli numbers have only simple poles at some roots
of unity, and their value at $q=1$ are the classical Bernoulli
numbers.

Let $P$ be a polynomial in $x$ with coefficients in $\QQ(q)$. Let us
call $q$-umbra of $P$ the value at $P$ of the $\QQ(q)$-linear form
which maps $x^n$ to the Carlitz $q$-Bernoulli number $\beta_n$. Let us
denote by $\Psi(P)$ the $q$-umbra of $P$. It is a rational function in
$q$.

Let $Q$ be an empty polytope. The polytope $\bop(Q)$, which is a
pyramid over $Q$, is also empty.

One has the following relations between the special value and the $q$-umbra.

\begin{proposition}
  One has
  \begin{equation}
    \sv_{\bop(Q),\bop(\lambda)} = \Psi( L_{Q,\lambda} ).
  \end{equation}
\end{proposition}
\begin{proof}
  The right hand-side is the action of a $\QQ(q)$-linear operator on a
  element of $\QQ(q)[x]$.

  By definition of the special value, mapping a multiple of $1 + q x$
  to its special value is a $\QQ(q)$-linear operator.

  By proposition \ref{delta_bop}, the left-hand side is obtained from
  $L_{Q,\lambda}$ by first applying the inverse of $\Delta$, then
  taking the special value. This is also a $\QQ(q)$-linear operator.

  It is therefore enough to check that this equality holds for enough
  polytopes, such that the $L_{Q,\lambda}$ span a basis of the space
  of polynomials. This has been done in the companion article
  \cite[\S 4]{serieqx} for the order polytopes of all tree posets.
\end{proof}

\begin{proposition}
  One has
  \begin{equation}
    \sv_{\bop\bop(Q),\bop\bop(\lambda)} = \Psi( -x L_{Q,\lambda} ).
  \end{equation}
\end{proposition}
\begin{proof}
  For the same reasons as before, both sides are $\QQ(q)$-linear
  operators acting on $L_{Q,\lambda}$. It is therefore enough to check
  it on enough polytopes.

  This has been done in the companion article \cite[\S 4]{serieqx} for the
  order polytopes of all tree posets.
\end{proof}

% \section{Open questions}

% Is there a symmetry of coefficients of the numerator of $q$-Ehrhart series
% for reflexive polytopes ? Known for $q=1$.

% en fait, ca caracterise les polytopes Gorenstein ! cf Benjamin Nill

% Order polytope of a pure poset is Gorenstein (Hibi, Stanley)

\bibliographystyle{plain}
\bibliography{q_ehrhart}

\begin{thebibliography}{10}

\bibitem{beck}
Matthias Beck and Sinai Robins.
\newblock {\em Computing the continuous discretely}.
\newblock Undergraduate Texts in Mathematics. Springer, New York, 2007.
\newblock Integer-point enumeration in polyhedra.

\bibitem{brion_ens}
Michel Brion.
\newblock Points entiers dans les poly\`edres convexes.
\newblock {\em Ann. Sci. \'Ecole Norm. Sup. (4)}, 21(4):653--663, 1988.

\bibitem{brion_bourbaki}
Michel Brion.
\newblock Points entiers dans les polytopes convexes.
\newblock {\em Ast\'erisque}, (227):Exp.\ No.\ 780, 4, 145--169, 1995.
\newblock S{\'e}minaire Bourbaki, Vol. 1993/94.

\bibitem{carlitz}
Leonard Carlitz.
\newblock {$q$}-{B}ernoulli numbers and polynomials.
\newblock {\em Duke Math. J.}, 15:987--1000, 1948.

\bibitem{serieqx}
Fr{\'e}d{\'e}ric Chapoton.
\newblock Sur une s{\'e}rie en arbres {\`a} deux param{\`e}tres.
\newblock preprint, 2013.

\bibitem{ehrhart}
Eug{\`e}ne Ehrhart.
\newblock Sur les poly\`edres rationnels homoth\'etiques \`a {$n$}\ dimensions.
\newblock {\em C. R. Acad. Sci. Paris}, 254:616--618, 1962.

\bibitem{feray_reiner}
Valentin F{\'e}ray and Victor Reiner.
\newblock {$P$}-partitions revisited.
\newblock {\em J. Commut. Algebra}, 4(1):101--152, 2012.

\bibitem{hahn}
Wolfgang Hahn.
\newblock \"{U}ber {O}rthogonalpolynome, die {$q$}-{D}ifferenzengleichungen
  gen\"ugen.
\newblock {\em Math. Nachr.}, 2:4--34, 1949.

\bibitem{proctor}
Robert~A. Proctor.
\newblock Bruhat lattices, plane partition generating functions, and minuscule
  representations.
\newblock {\em European J. Combin.}, 5(4):331--350, 1984.

\bibitem{stanley_ordered}
Richard~P. Stanley.
\newblock {\em Ordered structures and partitions}.
\newblock American Mathematical Society, Providence, R.I., 1972.
\newblock Memoirs of the American Mathematical Society, No. 119.

\bibitem{stanley_2p}
Richard~P. Stanley.
\newblock Two poset polytopes.
\newblock {\em Discrete Comput. Geom.}, 1(1):9--23, 1986.

\bibitem{stap2}
Alan Stapledon.
\newblock Weighted {E}hrhart theory and orbifold cohomology.
\newblock {\em Adv. Math.}, 219(1):63--88, 2008.

\bibitem{stap1}
Alan Stapledon.
\newblock Equivariant {E}hrhart theory.
\newblock {\em Adv. Math.}, 226(4):3622--3654, 2011.

\bibitem{stembridge}
John~R. Stembridge.
\newblock On minuscule representations, plane partitions and involutions in
  complex {L}ie groups.
\newblock {\em Duke Math. J.}, 73(2):469--490, 1994.

\end{thebibliography}

\end{document}